\documentclass[a4paper,11pt]{article}

\usepackage{fullpage}

\usepackage{amsmath, amsthm, amssymb, amsfonts, mathtools, mathrsfs}
\usepackage[shortlabels]{enumitem}

\usepackage{comment}

\usepackage[hidelinks]{hyperref}
\usepackage[nameinlink,capitalise,noabbrev]{cleveref}

\usepackage{etoolbox}
\patchcmd{\thebibliography}{\labelsep}{\labelsep\itemsep=2pt \parsep=1.5pt \relax}{}{}

\usepackage{graphicx}
\usepackage{float}
\usepackage{tikz}

\theoremstyle{definition}
\newtheorem{definition}{Definition}[section]
\newtheorem*{definition*}{Definition}

\newtheorem*{remark*}{Remark}

\theoremstyle{plain}

\newtheorem{theorem}[definition]{Theorem}
\newtheorem{lemma}[definition]{Lemma}
\newtheorem{proposition}[definition]{Proposition}
\newtheorem{claim}[definition]{Claim}

\newtheorem{corollary}[definition]{Corollary}
\newtheorem*{corollary*}{Corollary}

\newenvironment{poc}{\begin{proof}[Proof of claim]
}{\end{proof}}

\def \cl {\colon}
\def \ce {\coloneqq}
\def \E {\mathbb{E}}
\def \P {\mathbb{P}}

\renewcommand{\le}{\leqslant}
\renewcommand{\ge}{\geqslant}
\renewcommand{\leq}{\leqslant}
\renewcommand{\geq}{\geqslant}

\makeatletter

\renewcommand \b[2] {\binom{#1}{#2}}

\newcommand*{\rom}[1]{\expandafter\@slowromancap\romannumeral #1@}

\def \mC {\mathcal{C}}
\def \mD {\mathcal{D}}
\def \mF{\mathcal{F}}
\def \mG {\mathcal{G}}
\def \mH{\mathcal{H}}

\def \N {\mathbb{N}}
\def \R {\mathbb{R}}

\def \l {\left}
\def \r {\right}

\newcommand*{\arXiv}[1]{\href{http://arxiv.org/pdf/#1}{arXiv:#1}}

\newcommand{\dst}{d_{\text{st}}}
\newcommand{\dnd}{d_{\text{nd}}}

\newcommand{\fert}{f_{\alpha_s; \text{\tiny$\underbrace{K_3, \cdots, K_3}_t$}}}
\newcommand{\ferj}{f_{\alpha_s; \text{\tiny$\underbrace{K_3, \cdots, K_3}_j$}}}
\newcommand{\fertg}{f_{\alpha_s; \text{${K_{i_1}, \cdots, K_{i_t}}$}}}

\title{Multicolor Erd\H{o}s--Rogers Functions}
\author{
	Hong Liu\thanks{Extremal Combinatorics and Probability Group (ECOPRO), Institute for Basic Science (IBS), Daejeon, South Korea. Supported by the Institute for Basic Science (IBS-R029-C4). \texttt{hongliu@ibs.re.kr}. }
	\and
    Haoran Luo\thanks{Department of Mathematics, Statistics and Computer Science, University of Illinois Chicago, Chicago, Illinois, USA. Research was partially performed while the second author was at the University of Illinois Urbana-Champaign and supported in part by Dr. James J. Woeppel Fellowship. \texttt{haoranl8@uic.edu}. }
	\and
	Minghui Ouyang\thanks{School of Mathematical Sciences, Peking University, Beijing 100871, China. \texttt{ouyangminghui1998@gmail.com}. }
}
\date{}

\begin{document}
\maketitle
\begin{abstract}
	In this paper, we study a multicolor variant of Erd\H{o}s--Rogers functions. Let $f_{\alpha_s; K_{i_1}, \cdots, K_{i_t}}(n)$ be the largest integer $m$ such that there is always an induced $K_s$-free subgraph of size $m$ in every $n$-vertex graph with a $t$-edge-coloring in which the edges with the $j$-th color induce no copy of $K_{i_j}$. We establish both upper and lower bounds for this multicolor version. Specifically, we show that $f_{\alpha_5; K_3, K_3}(n) = n^{1/2+o(1)}$, $\Omega(n^{5/11}) \le f_{\alpha_5; K_3, K_3, K_3}(n) \le n^{1/2+o(1)}$, and $\Omega(n^{20/61}) \le f_{\alpha_5; K_3, K_3, K_3, K_3}(n) \le n^{1/3+o(1)}$. 
\end{abstract}

\section{Introduction} \label{sec:int}
The Ramsey number $R(a,b)$ is defined as the smallest integer $n$ such that every red-blue edge-coloring of $K_n$, the complete graph on $n$ vertices, contains a red $K_a$ or blue $K_b$. Determining its asymptotic growth has attracted significant attention in Extremal Combinatorics~\cite{Ram29, ES35, Spe77, AKS80, Kim95, FGM20, MV24, CJMS25+}. When $a$ is fixed and $b$ varies, $R(a,b)$ is referred to as the off-diagonal Ramsey number. In this regime, $R(a,b)$ grows polynomially with respect to $b$.

Erd\H{o}s and Rogers~\cite{ER62} in 1962 introduced the following natural generalization of the Ramsey number:
\begin{equation*}
    f_{s,t}(n) \ce \min \Bigl\{ \alpha_s(G) \cl G \text{ is a $K_t$-free graph with $n$ vertices} \Bigr\}.
\end{equation*}
Here, $\alpha_s(G)$ denotes the \emph{$s$-independence number} of a graph $G$, i.e., the maximum size of a vertex subset that induces no copy of $K_s$. Note that $f_{2,t}$ is precisely the inverse function of the Ramsey function: $f_{2,t}(n) \ge b \iff R(t,b) \le n$.

The function $f_{s,t}$ is now known as the \emph{Erd\H{o}s--Rogers function}, and a central problem is to determine its growth as $n \to \infty$ for various fixed values of $s$ and $t$. The case $t = s+1$ has been extensively studied and is now resolved up to a polylogarithmic factor through a series of works~\cite{ER62, DR11, Wol13, DRR14, MV25}. The growth of the next case, $t = s+2$, however, remains open~\cite{Sud05a, JS25}. For general values of $s$ and $t$, earlier results appear in~\cite{BH91, Kir94}. The best-known lower bound on $f_{s,t}(n)$ is given by a recursive formula due to Sudakov~\cite{Sud05a, Sud05b}, while the best-known upper bound was established by Krivelevich~\cite{Kri95}. We refer the readers to~\cite{JS25} and the citations therein for further discussion of the Erd\H{o}s--Rogers function. Recently, some generalized Erd\H{o}s--Rogers functions for arbitrary pairs of graphs have also been considered~\cite{BCL25, MV24b+, GJS25}.

In this paper, we investigate a multicolor version of the Erd\H{o}s--Rogers function. For integers $i_1,\cdots, i_t$, we say that a graph $G$ is \emph{$(K_{i_1},\cdots, K_{i_t})$-free} if there is a coloring of $E(G)$ with $t$ colors such that the edges with the $j$-th color induce no copy of $K_{i_j}$ for every $1\le j \le t$.
The \emph{multicolor Ramsey number} $R(i_1,\cdots,i_t)$ is the smallest integer $n$ such that $K_n$ is not $(K_{i_1},\cdots, K_{i_t})$-free.
We refer the readers to~\cite{HW20, CF21, Wig21, Saw22, BBCGHMST24+} for some known bounds on the multicolor Ramsey numbers. We write $r_{t} (b)$ for $R(\underbrace{b,b,\cdots,b}_t)$.

In the same spirit as the multicolor Ramsey number generalizes the usual Ramsey number, we define the following:
\begin{definition*}[Multicolor Erd\H{o}s--Rogers function]
	Let $s$, $t$, and $i_1,\cdots,i_t$ be positive integers. Define
	   \[ \fertg (n) \ce \min \Bigl\{ \alpha_s(G) \cl G \text{ is a $(K_{i_1},\cdots, K_{i_t})$-free graph with $n$ vertices} \Bigr\}. \]
\end{definition*}

When $s=2$, the multicolor Erd\H{o}s--Rogers function can be viewed as the inverse of the multicolor Ramsey number. A notable result in this direction is a theorem of Alon and R\"{o}dl~\cite{AR05}, who proved that\footnote{Here $\widetilde{\Theta}$ means $\Theta$ up to a polylogarithmic factor. } $R(\underbrace{3,\cdots,3}_t, m) = \widetilde{\Theta}(m^{t+1})$ or equivalently $f_{\alpha_2; \text{\tiny$\underbrace{K_3, \cdots, K_3}_t$}} (n)= \widetilde{\Theta}(n^{\frac{1}{t+1}})$. This result was extended by He and Wigderson~\cite{HW20}. 

We remark on the following notable difference between the problems of the multicolor Ramsey numbers and the multicolor Erd\H{o}s--Rogers function. Take the example of $R(3,3,m)$, or equivalently $f_{\alpha_2,K_3,K_3}(n)$. The idea in~\cite{AR05, HW20} is to start with a triangle-free graph with moderately number of edges, and then put together two copies of this triangle-free graph randomly. However, if we now consider $f_{\alpha_3,K_3,K_3}(n)$, then by definition, each color class contains no triangle at all and hence every subset of vertices is trivially triangle-free. Therefore, in order to fully conquer this problem, it is not enough to consider each color class separately and we are forced to consider the triangles formed by edges from both color classes. This fact leads to some additional nuance in the analysis of the multicolor Erd\H{o}s--Rogers function.

Our first result is a construction that provides the following upper bounds.

\begin{theorem} \label{thm:upper_bound}
	For integers $s \ge 2$, $t \ge 1$, and $b \ge 3$, we have
	\[ f_{\alpha_s; \text{\tiny$\underbrace{K_b, \cdots, K_b}_t$}} (n) \le n^{1/\left( \lfloor \frac{t}{\ell} \rfloor + 1 \right) + o(1)}, \]
	where $\ell$ is the minimum integer such that $r_\ell(b) > s$.
\end{theorem}

Heuristically, the exponent of $n$ is determined by $\lfloor \frac{t}{\ell} \rfloor$, the maximum number of disjoint color classes of size $\ell$ such that the $\ell$-color Ramsey number of $K_b$ exceeds $s$. Take the simplest case, for example when $\lfloor \frac{t}{\ell} \rfloor=0$, then $t<\ell$ and $r_t(b)\le s$, and any $(K_b,\cdots,K_b)$-free graph must be $K_s$-free, yielding a trivial bound $n$. It is worth noting that at the transition point, namely when $\lfloor \frac{t}{\ell} \rfloor$ increases from $0$ to $1$, the bound drops from $n$ to $n^{1/2+o(1)}$. If we take $(s,b,\ell)_{\ref{thm:upper_bound}}=(2,3,1)$ in \cref{thm:upper_bound}, then we recover the aforementioned result of Alon and Rödl~\cite{AR05}.

The upper bound in~\cref{thm:upper_bound} is optimal for an infinite class of parameters. Together with \Cref{thm:lower_bound} below, we have the following.

\begin{corollary} \label{cor:examples}
	Let $t\in\mathbb{N}$ and $s$ be such that $r_{t-1}(3) \le s < r_t(3)$. Then we have
	$$\fert (n) =n^{1/2+o(1)}.$$
	Furthermore, we have
	\begin{gather*}
		f_{\alpha_3; K_3, K_3}(n) \,=\, f_{\alpha_4; K_3, K_3}(n) \,=\, f_{\alpha_5; K_3, K_3}(n) \,=\, n^{1/2+o(1)}, \\
		\Omega(n^{5/11}) \,\le\, f_{\alpha_5; K_3, K_3, K_3}(n) \,\le\, n^{1/2+o(1)}, \\
		\Omega(n^{20/61}) \,\le\, f_{\alpha_5; K_3, K_3, K_3, K_3}(n) \,\le\, n^{1/3+o(1)}.
	\end{gather*}
\end{corollary}
Note that we have a fairly good understanding of the 2-color case $f_{\alpha_s;K_3,K_3}$. Indeed, $f_{\alpha_2; K_3, K_3}(n)=n^{1/3+o(1)}$, $f_{\alpha_s; K_3, K_3}(n)=n^{1/2+o(1)}$ for $3\le s\le 5$, and $f_{\alpha_s; K_3, K_3}(n)=n$ for all $s\ge 6$.

Motivated by the recursive lower bound of Sudakov~\cite{Sud05a, Sud05b}, we also prove the following lower bounds. To state our recursive formula, we need the notion of the \emph{local Ramsey number}~\cite{GLST87}, a variant of the classical Ramsey number that arises in the study of edge-colorings with local constraints. Let $H$ be a fixed graph and $k$ be a positive integer. The local Ramsey number $r^{\text{loc}}_k(H)$ is defined as the smallest integer $n$ such that the following holds: in every edge-coloring of the complete graph $K_n$ using colors from $\N$ such that each vertex is incident to edges of at most $k$ distinct colors, there exists a monochromatic copy of $H$. We use the shorthand $r^{\text{loc}}_k(3)$ for the local Ramsey number when $H = K_3$. Note that $r_k^{\text{loc}}(3) \ge r_k(3)$. Let $g \cl \N \to \N$ be the inverse function of the local Ramsey number, that is, $g(i)$ is the smallest integer $k$ such that $r^{\text{loc}}_k(3) > i$. Since it can be checked that $r^{\text{loc}}_0(3)=2$, $r^{\text{loc}}_1(3)=3$, and $r^{\text{loc}}_2(3)=6$, we have the following small values of $g$: $g(2)=1$, $g(3)=g(4)=g(5)=2$, $g(6)=3$.

\begin{theorem} \label{thm:lower_bound}
	For every integer $s \ge 2$, we have
	   \[ \fert (n) \ge \Omega(n^{a_t}), \]
	where the sequence $\{a_t\}$ is defined recursively by
	\begin{equation*}
		a_t =
		\left\{
		\begin{array}{ll}
			1                                                                              & \textrm{if $r_t(3) \le s$},                                \\
			\frac{1}{2}                                                                    & \textrm{if $r_{t-1}(3) \le s < r_t(3)$},                   \\
			\frac{s + \lceil \frac{s}{2} \rceil - 3}{2s + 2 \lceil \frac{s}{2} \rceil - 5} & \textrm{if $r_{t-2}(3) \le s < r_{t-1}(3)$ and $s \ge 3$},
		\end{array}
		\right.
	\end{equation*}
	and
	   \[ \frac{1}{a_t} = 1 + \frac{1}{s-1} \sum_{i = 2}^{s} \frac{1}{a_{t-g(i)}} \quad \textrm{for $3 \le s < r_{t-2}(3)$ or $2 = s < r_{t-1}(3)$.} \]
\end{theorem}

\medskip

\noindent\textbf{Notation and Organization.}
Throughout the paper, we assume that $s \ge 2$ is a fixed integer and $n$ is sufficiently large with respect to $s$. All logarithms are taken to base~$2$. In \Cref{sec:blk_constr}, we present the block construction and use it to prove \Cref{thm:upper_bound} in \Cref{sec:upper_bound}. \Cref{thm:lower_bound} is established in \Cref{sec:lower_bound}. We conclude with some remarks in \Cref{sec:con}.

\section{The Block Construction} \label{sec:blk_constr}
Our approach is inspired by recent developments on the Erd\H{o}s--Rogers function~\cite{DRR14, MV24, JS25, MV24b+, MV25}. We apply the Hypergraph Container Lemma~\cite{ST15, BMS15} to bound the number of large $K_s$-free subsets. Taking advantage of the multicolor setting, we also use an idea from Alon and R\"{o}dl~\cite{AR05} in their study of multicolor Ramsey numbers, where they randomly overlaid multiple copies of a basic construction to obtain the final one. 

As remarked in \cref{sec:int}, it is not enough to simply take some triangle-free graph, since that would give a trivial upper bound. 
A building block in our construction is a triangle-free hypergraph introduced by Erd\H{o}s, Frankl, and R\"{o}dl~\cite{EFR86}, which was later adapted by Mubayi and Verstra\"ete in their study of the Erd\H{o}s--Rogers function for arbitrary pairs of graphs~\cite{MV24b+}.

\begin{proposition}[{\cite[Theorem 4]{MV24b+}}] \label{prop:EFR_constr}
	For any integers $R, n \ge 3$ with $n \ge R \ge \log n$, there exists an $n$-vertex $R$-uniform hypergraph $\mH$ with the following properties:
	\begin{enumerate}[label=(\roman*), align=left, labelsep=0.8em, leftmargin=2.5em, rightmargin=2em, itemsep=-0.3ex, topsep=0.3ex]
		\item $|E(\mH)| \geq n^2/R^{8\sqrt{\log_R n}}$.
		\item $\mH$ is linear, that is, for any distinct edges $e,f \in \mH$, $|e \cap f| \leq 1$.
		\item $\mH$ is triangle-free, that is, for any three distinct edges $e,f,g \in \mH$, if $|e \cap f| = |f \cap g| = |g \cap e| = 1$ then $|e \cap f \cap g| = 1$.
	\end{enumerate}
\end{proposition}

We apply \Cref{prop:EFR_constr} with $R = C \log^2 n$, where $C = 32s(s+1)^2$. Then, the condition~(i) yields
    \[ |E(\mH)| \geq \frac{n^2}{R^{8\sqrt{\log_R n}}} \ge \frac{n^2}{(\log n)^{32\sqrt{\log n}}} \quad \text{for } n \ge C(s), \]
where $C(s)$ is a sufficiently large constant depending only on $s$. For simplicity, we omit rounding and assume that $R$ is an integer.

We define the edge-vertex incidence graph $G$ of $\mH$ by setting $V(G) \ce E(\mH)$ and $E(G) \ce \{ \{e,f\} \cl  e, f \in E(\mH) ,\, e \ne f,\, e \cap f \ne \varnothing\}$. For each vertex $v \in V(\mH)$, the set of edges containing $v$ forms a clique $K_v = \{e \in E(\mH)\cl v \in e\}$ in $G$. Due to the linearity of $\mH$, any two such cliques intersect in at most one vertex (of $G$). Hence, $E(G)$ is the edge-disjoint union of $K_v$ for all $v \in V(\mH)$. Also note that by \cref{prop:EFR_constr} (iii), every triangle in $G$ lies entirely within some clique $K_v$ for $v \in V(\mH)$ and so is any larger clique.

We let $G_*$ denote an \emph{$s$-partite sparsification of $G$}, that is, a graph obtained from $G$ by replacing each clique $K_v$ with a complete $s$-partite graph. We say that a subset $X \subseteq V(G)$ is \emph{evenly partitioned} in a clique $K_v$ if its intersection with each partite class in the $s$-partition of $K_v$ has size at least $\frac{|X \cap K_v|}{s+1}$. For any subset $X \subseteq V(G)$, define $a_v \ce |X \cap K_v|$ for each vertex $v \in V(\mH)$, and let $I_i \ce \{ v \in V(\mH) \cl 2^{i-1} \le a_v < 2^i \}$. Let $\ell \in \{1,\cdots, \log n\}$ be the index maximizing $\sum_{v \in I_\ell} a_v$.
The following lemma ensures the existence of an $s$-partite sparsification of $G$ in which every subset $X \subseteq V(G)$ with $|X| \ge n$ is evenly partitioned in at least half of the cliques that contribute significantly to it.

\begin{lemma} \label{lem:good_partition}
	There exists an $s$-partite sparsification $G_*$ of $G$ such that for every subset $X \subseteq V(G_*) = V(G)$ with $|X| \ge n$, the set $X$ is evenly partitioned in at least half of the cliques $K_v$ with indices $v \in I_\ell$.
\end{lemma}

\begin{proof}
	We construct $G_*$ by sparsifying each clique $K_v$ independently: for each $v$, we assign a random function $f_v \cl K_v \to [s]$, and retain only those edges $\{x, y\} \subseteq K_v$ with $f_v(x) \ne f_v(y)$.  By applying a Chernoff-type bound for Bernoulli trials and taking a union bound over the $s$ parts, we obtain
	   \[ \P \Bigl(X \text{ is not evenly partitioned in } K_v \Bigr) \le s \cdot \exp \left( -\frac{|X \cap K_v|}{2s(s+1)^2} \right). \]

	Since $\mH$ is a $(C \log^2 n)$-graph and $G$ is its edge-vertex incidence graph, each vertex in $G$ belongs to exactly $C \log^2 n$ cliques $K_v$. Hence,
	   \[ \sum_{v \in V(G)} a_v = C \log^2 n \cdot |X|. \]
	Moreover, as $\mH$ is linear, $a_v \le |K_v| \le  n$ for all $v \in V(G)$. By the maximality of $\sum_{v \in I_\ell} a_v$ over all $I_i$, we have
	   \[ \sum_{v \in I_\ell} a_v \ge \frac{1}{\log n} \sum_{v \in V(G)} a_v = C \log n \cdot |X|, \]
	which implies
	\begin{equation} \label{eq:I_ell}
		\frac{C \log n \cdot |X|}{2^\ell} \le |I_\ell| \le \frac{C \log^2 n \cdot |X|}{2^{\ell-1}}.
	\end{equation}
	Therefore,
	\begin{align*}
		  & \P \Bigl(\text{at least half of } K_v \text{ for } v \in I_\ell \text{ are not evenly partitioned} \Bigr)                 \\
		  & \le \sum_{\substack{I \subseteq I_\ell                                                                                    \\ |I| = |I_\ell|/2}} \prod_{v \in I} \left( s \cdot \exp \left( -\frac{|X \cap K_v|}{2s(s+1)^2} \right) \right) \\
		  & \le \b{|I_\ell|}{|I_\ell|/2} \cdot s^{|I_\ell|/2} \cdot \exp \l( -\frac{|I_\ell|}{4s(s+1)^2} \cdot 2^{\ell-1} \r)         \\
		  & \le 2^{|I_\ell|} \cdot s^{|I_\ell|} \cdot \exp \l( -\frac{|I_\ell|}{4s(s+1)^2} \cdot 2^{\ell-1} \r)                       \\
		  & \le (2s)^n \cdot \exp \l( -\frac{C \log n \cdot |X|}{2^{\ell+2} s (s+1)^2} \cdot 2^{\ell-1} \r) \tag{by~\eqref{eq:I_ell}} \\
		  & = (2s)^n \cdot \exp \l( -\frac{C \log n \cdot |X|}{8 s (s+1)^2} \r)                                                       \\
		  & \le \exp \l( -\frac{C \log n \cdot |X|}{16s(s+1)^2} \r) \tag{$|X| \ge n$ and $n \ge C(s)$}
	\end{align*}

	We now apply a union bound over all subsets $X \subseteq V(G)$ with $|X| \ge n$. The total failure probability is at most
	   \[ \sum_{k = n}^{|V(G)|} \b{|V(G)|}{k} \cdot \exp \l( -\frac{C \log n \cdot k}{16s(s+1)^2} \r) < \sum_{k = n}^{|V(G)|} \l( \frac{e |V(G)|}{k} \r)^k \cdot \exp \l( -\frac{C \log n \cdot k}{16s(s+1)^2} \r) < \sum_{k = n}^{|V(G)|} 2^{-k} < 1, \]
	where we used the trivial estimate $|V(G)| \le n^2$ as $\mH$ is linear and $C = 32 s (s+1)^2$.

	Therefore, with positive probability, the sparsification $G_*$ satisfies the desired property.
\end{proof}

In \cref{sec:upper_bound}, we will use the Hypergraph Container Lemma to bound the number of $K_s$-free subsets in the $G_*$ given by \cref{lem:good_partition}. For this purpose, we introduce the following notion of \emph{uniformly dense} hypergraphs. For a hypergraph $\mG$ and a positive integer $i$, let $\Delta_i(\mG)$ be the maximum number of edges containing any fixed set of $i$ vertices, and let $v(\mG)$ and $e(\mG)$ denote the number of vertices and edges in $\mG$, respectively.

\begin{definition} \label{def:uniformly_dense}
	Suppose $N, m \in \mathbb{N}$ and $\alpha, \lambda \in \mathbb{R}_+$. An $s$-uniform hypergraph $\mG$ on $N$ vertices is said to be \emph{$(N,m,\alpha,\lambda)$-uniformly dense} if, for every subset $X \subseteq V(\mG)$ with $|X| \ge m$, there exists a subgraph $\mG' \subseteq \mG[X]$ with vertex set $X$ such that
	   \[ e(\mG') \ge \alpha \cdot |X|^s, \]
	and
	   \[ \forall i \in [s],\ \Delta_i(\mG') \le \lambda \left( \frac{e(\mG')}{|X|} \right)^{1 - \frac{i-1}{s-1}}. \]
\end{definition}

\begin{lemma} \label{lem:good_uniformly_dense}
	Suppose $G_*$ is an $s$-partite sparsification of $G$ given by \Cref{lem:good_partition}. Consider the $s$-uniform hypergraph $\mG$ on $V(G)$ whose edge set consists of all $s$-subsets $\{v_1, \cdots, v_s\}$ that form a copy of $K_s$ in $G_*$. Then $\mG$ is $(v(\mG), n, \frac{1}{n^{s-1}}, O_s(\log n))$-uniformly dense.
\end{lemma}

\begin{proof}
	Let $X \subseteq V(\mG) = V(G)$ be any subset with $|X| \ge n$. Define  $\{a_v\}$, $\{I_i\}$, and $\ell$ for $X$ as before. By the construction of $G_*$, the set $X$ is evenly partitioned in at least half of the cliques $K_v$ with $v \in I_\ell$. Let $I_\ell' \subseteq I_\ell$ denote the subset of such cliques.
	Define $\mG' \subseteq \mG[X]$ to be the subgraph consisting of only those $K_s$-copies in $G_*[X]$ that are contained in the partitioned cliques $K_v$ for some $v \in I_\ell'$. We verify that $\mG'$ witnesses the uniform density for $\mG$.

	\vspace{-2.5ex}\paragraph{Edge counting.} \leavevmode\\
	From $\sum_{v \in V(G)} a_v = C \log^2 n \cdot |X|$, we deduce that $2^\ell \ge \frac{C \log n \cdot |X|}{n}$, for otherwise we would have
	   \[ \sum_{v \in V(G)} a_v \le \log n \cdot \sum_{v \in I_\ell} a_v \le \log n \cdot n \cdot 2^\ell < C \log^2 n \cdot |X| = \sum_{v \in V(G)} a_v, \]
	a contradiction.

	For each evenly partitioned clique $K_v$, selecting one arbitrary vertex from each part yields at least $\left( \frac{a_v}{s+1} \right)^s$ distinct copies of $K_s$. Since $a_v \ge 2^{\ell-1}$ for all $v \in I_\ell$, and at least half of the cliques with indices in $I_\ell$ are evenly partitioned, we obtain
	\begin{equation} \label{eq:edge_G'}
		e(\mG') \ge \sum_{v \in I_\ell'} \l( \frac{a_v}{s+1} \r)^s \ge \frac{|I_\ell|}{2} \cdot \left( \frac{2^{\ell-1}}{s+1} \right)^s \ge \frac{C \log n \cdot |X|}{2^{s+1}} \cdot \frac{2^{(s-1)\ell}}{(s+1)^s},
	\end{equation}
	where the last inequality uses~\eqref{eq:I_ell}. Since $2^\ell \ge \frac{C \log n \cdot |X|}{n}$, it follows that
	   \[ e(\mG') \ge \frac{C \log n \cdot |X|}{2^{s+1}(s+1)^s} \cdot \l( \frac{C \log n \cdot |X|}{n} \r)^{s-1} = \frac{C^s \log^{s} n}{2^{s+1} (s+1)^s} \cdot \frac{|X|^s}{n^{s-1}} \ge \frac{1}{n^{s-1}} \cdot |X|^s. \]

	\vspace{-2.5ex}\paragraph{Codegree conditions.} \leavevmode\\
	Let $\lambda = O_s (\log n)$. We show that for all $i \in [s]$, the codegree satisfies $\Delta_i(\mG') \le \lambda \left( \frac{e(\mG')}{|X|} \right)^{1 - \frac{i-1}{s-1}}$.

	For $i = 1$, each vertex of $G_*$ lies in $C \log^2 n$ cliques, and each such clique contributes at most $\left( \frac{2^\ell}{s-1} \right)^{s-1}$ many copies of $K_s$ containing that vertex. Using~\eqref{eq:edge_G'} to estimate $2^{(s-1)\ell}$, we obtain
	\begin{align*}
		\Delta_1(\mG') & \le C \log^2 n \cdot \left( \frac{2^\ell}{s-1} \right)^{s-1} \le \frac{C \log^2 n}{(s-1)^{s-1}} \cdot \frac{2^{s+1}(s+1)^s \cdot e(\mG')}{C \log n \cdot |X|} \le \lambda \cdot \frac{e(\mG')}{|X|}.
	\end{align*}

	For every $i \ge 2$, since any set of $i$ vertices is contained in at most one common clique, and each such clique contributes at most $\left( \frac{2^\ell}{s-i} \right)^{s-i}$ copies of $K_s$ containing those $i$ vertices. Hence,
	\begin{align*}
		\Delta_i(\mG') & \le \left( \frac{2^\ell}{s-i} \right)^{s-i} = \frac{1}{(s-i)^{s-i}} \cdot \left( 2^{(s-1)\ell} \right)^{1 - \frac{i-1}{s-1}}                                                                          \\
		               & \le \frac{1}{(s-i)^{s-i}} \cdot \left( \frac{2^{s+1} (s+1)^s}{C \log n} \cdot \frac{e(\mG')}{|X|} \right)^{1 - \frac{i-1}{s-1}} \le \lambda \left( \frac{e(\mG')}{|X|} \right)^{1 - \frac{i-1}{s-1}}.
	\end{align*}
	This completes the proof.
\end{proof}

The following lemma shows that the property of uniform density is preserved under the blow-up operation, up to a polylogarithmic loss in the parameters.

\begin{lemma} \label{lem:blow_up}
	Let $N, n, k \in \N$ and let $\alpha, \lambda \in \R_+$. If an $s$-uniform hypergraph $\mG$ is $(N, n, \alpha, \lambda)$-uniformly dense, then its $k$-blow-up $\mG^{(k)}$ is $\left( kN, k (\lceil \log k \rceil + 1) n,  \frac{\alpha}{2^s (\lceil \log k \rceil + 1)^s}, 2^s (\lceil \log k \rceil + 1) \lambda \right)$-uniformly dense.
\end{lemma}

\begin{proof}
	Clearly $v(\mG^{(k)}) = k \cdot v(\mG) = kN$. We verify that $\mG^{(k)}$ satisfies the conditions of uniform density.

	Let $U \subseteq V(\mG^{(k)})$ be an arbitrary subset with $|U| \ge k (\lceil \log k \rceil + 1) n$. For each vertex $u \in V(\mG^{(k)})$, define its projection $\pi(u) \in V(\mG)$ to be the original vertex in $\mG$ to which $u$ corresponds. Partition $U = U_1 \sqcup U_2 \sqcup \cdots \sqcup U_{\lceil \log k \rceil + 1}$, where
	   \[ U_i = \Bigl\{ u \in U \cl 2^{i-1} \le \left| \pi^{-1}(\pi(u)) \cap U \right| < 2^i \Bigr\}. \]

	Take an index $t$ such that $|U_t|=\max_i |U_i| \ge kn$. Let $X = \pi(U_t)$, then $|X| \ge n$ and $|U| \le (\lceil \log k \rceil + 1) |U_t| \le (\lceil \log k \rceil + 1)2^t \cdot |X|$. Let $\mG' \subseteq \mG[X]$ be the subgraph guaranteed by the uniform density of $\mG$, that is,
	   \[ e(\mG') \ge \alpha |X|^s, \qquad \text{and} \qquad \Delta_i(\mG') \le \lambda \left( \frac{e(\mG')}{|X|} \right)^{1 - \frac{i-1}{s-1}} \quad \text{for all } i \in [s]. \]

	Now consider the subgraph $\mH \subseteq \mG^{(k)}[U]$ consisting of all $s$-sets lifted from the edges of $\mG'$, by selecting one vertex from each fiber of $\pi^{-1}(x)$. Then $e(\mH) \ge (2^{t-1})^s \cdot e(\mG')$. We verify that $\mH$ witnesses the uniform density for $\mG$.

	\vspace{-2.5ex}\paragraph{Edge counting.} \leavevmode\\
	   \[ e(\mH) \ge (2^{t-1})^s \cdot e(\mG') \ge 2^{(t-1)s} \cdot \alpha |X|^s \ge \frac{\alpha}{2^s} \cdot (2^t |X|)^{s} \ge \frac{\alpha}{2^s (\lceil \log k \rceil + 1)^s} \cdot |U|^s. \]

	\vspace{-3ex}\paragraph{Codegree conditions.} \leavevmode\\
	Fix $i \in [s]$. Each edge in $\mH$ corresponds to an edge in $\mG'$, and the number of such lifted edges containing any fixed set of $i$ vertices is at most $2^{(s-i)t} \cdot \Delta_i(\mG')$. Therefore,
	\begin{align*}
		\Delta_i(\mH) & \le 2^{(s-i)t} \cdot \Delta_i(\mG') \le 2^{(s-i)t} \cdot \lambda \left( \frac{e(\mG')}{|X|} \right)^{1 - \frac{i-1}{s-1}} = \lambda \left( \frac{2^{(s-1)t} \cdot e(\mG')}{|X|} \right)^{1 - \frac{i-1}{s-1}}                     \\
		              & \le \lambda \left( \frac{(\lceil \log k \rceil + 1) \cdot 2^{st} \cdot e(\mG')}{|U|} \right)^{1 - \frac{i-1}{s-1}} \le \lambda \left( \frac{(\lceil \log k \rceil + 1) \cdot 2^s \cdot e(\mH)}{|U|} \right)^{1 - \frac{i-1}{s-1}} \\
		              & \le ((\lceil \log k \rceil + 1) \cdot 2^s \cdot \lambda) \cdot \left( \frac{e(\mH)}{|U|} \right)^{1-\frac{i-1}{s-1}}. \qedhere
	\end{align*}
\end{proof}

\section{Upper Bounds for Multicolor Erd\H{o}s--Rogers Functions} \label{sec:upper_bound}
In this section, we build upon the construction from \Cref{sec:blk_constr} and apply a suitable blow-up to derive upper bounds on multicolor Erd\H{o}s--Rogers functions. The polylogarithmic factors in the following corollary are not optimized, in order to streamline its statement and applications.

We use the following version of the Hypergraph Container Lemma from~\cite{JS25}. Iterated applications of this lemma yield bounds on the number of independent sets in an $s$-uniform hypergraph.

\begin{lemma}[{\cite [Corollary 1]{JS25}}] \label{lem:BMS_container}
	For every positive integer $s\ge 2$ and positive reals $p$ and $\lambda$, the following holds. Suppose that $\mG$ is a non-empty $s$-uniform hypergraph with at least two vertices such that $pv(\mG)$ and $v(\mG)/\lambda$ are integers, and for every $i \in [s]$,
	   \[ \Delta_i(\mG) \le \lambda\cdot p^{i-1} \frac{e(\mG)}{v(\mG)}. \]

	Then there exists a collection $\mC$ of at most $v(\mG)^{spv(\mG)}$ sets of size at most $(1-\delta \lambda^{-1})v(\mG)$ such that for every independent set $I \subseteq V(\mG)$, there exists some $R \in \mC$ with $I \subseteq R$, where $\delta=2^{-s(s+1)}$.
\end{lemma}
Recall that we defined the uniformly dense hypergraphs in \cref{def:uniformly_dense}. It is straightforward to verify that for any $(N,m,\alpha,\lambda)$-uniformly dense hypergraph $\mG$ and any subset $X \subseteq V(\mG)$ with $|X| \ge m$, the witness $\mG' \subseteq \mG[X]$ for the uniform density satisfies the conditions of \Cref{lem:BMS_container} by taking $p = \left( \alpha^{\frac{1}{s-1}} \cdot |X| \right)^{-1}$ and using the same $\lambda$.
Indeed, $\mG'$ is obviously non-empty, and for each $i \in [s]$,
\begin{align*}
	\Delta_i(\mG') & \le \lambda \left( \frac{e(\mG')}{v(\mG')} \right)^{1-\frac{i-1}{s-1}} = \lambda \cdot \left(\frac{e(\mG')}{v(\mG')}\right)^{-\frac{i-1}{s-1}} \cdot \frac{e(\mG')}{v(\mG')}                                                       \\
	               & \le \lambda \cdot \left(\frac{\alpha \cdot v(\mG')^s}{v(\mG')}\right)^{-\frac{i-1}{s-1}} \cdot \frac{e(\mG')}{v(\mG')} = \lambda \cdot \left( \alpha^{\frac{1}{s-1}} \cdot v(\mG') \right)^{-(i-1)} \cdot \frac{e(\mG')}{v(\mG')}.
\end{align*}

\begin{corollary} \label{cor:count}
	Let $G_*$ be an $s$-partite sparsification given by \Cref{lem:good_partition}, and $k \ge 1$ is an integer.
	Consider the $k$-blow-up $G_*^{(k)}$ of $G_*$. Then there exists a collection $\mC$ of at most $(10k)^t$ subsets of size at most $t = (\lceil \log k \rceil + 1)^2 \cdot n \log^4 n$, such that every $K_s$-free subset of $G_*^{(k)}$ is contained in some $R \in \mC$.
\end{corollary}

\begin{proof}
	Suppose $\mG$ and $\mG^{(k)}$ are the $s$-uniform hypergraphs whose edge set consists of all $s$-subsets $\{v_1, \cdots, v_s\}$ that form a copy of $K_s$ in $G_*$ and $G_*^{(k)}$ respectively. It is clear that $\mG^{(k)}$ is also the $k$-blow-up of $\mG$. By \Cref{lem:good_partition}, \Cref{lem:good_uniformly_dense}, and \Cref{lem:blow_up},
	   \[ \mG^{(k)} \text{ is } \Big( kN, k (\lceil \log k \rceil + 1) n,  \frac{1}{2^s (\lceil \log k \rceil + 1)^s n^{s-1}}, 2^s (\lceil \log k \rceil + 1) \lambda \Big) \text{-uniformly dense}. \]
	Therefore, by \Cref{lem:BMS_container} and the discussion below \Cref{def:uniformly_dense}, for every subset $U \subseteq V(\mG^{(k)}) = V(G_*^{(k)})$ of size at least $k (\lceil \log k \rceil + 1) n$, there exists a collection $\mC(U)$ of at most $|U|^{sp|U|}$ sets of size at most $(1-\delta\lambda^{-1})|U|$ such that every $K_s$-free subset of $U$ is contained in some element $R \in \mC(U)$, where $\lambda = O_s(\log n)$, $\delta = 2^{-s(s+1)}$, and $p = \frac{\left( 2^s (\lceil \log k \rceil + 1)^s n^{s-1} \right)^{1/(s-1)}}{|U|} \le \frac{2n (\lceil \log k \rceil + 1)^2}{|U|}$.

	Let
	   \[ \mC_0 = \left\{ V(G_*^{(k)}) \right\},\quad \mD_0 = \varnothing. \]
	Suppose $\mC_j$ and $\mD_j$ are defined, we then define
	\begin{gather*}
		\mC_{j+1} = \bigcup_{U \in \mC_j} \left\{ X \in \mC(U) \cl |X| \ge k (\lceil \log k \rceil + 1) n \right\}, \\
		\mD_{j+1} = \mD_j \cup \bigcup_{U \in \mC_j} \left\{ X \in \mC(U) \cl |X| < k (\lceil \log k \rceil + 1) n \right\}.
	\end{gather*}

	It is easy to see that each element of $\mC_j$ has size at least $k (\lceil \log k \rceil + 1) n$. Note that $|U| \le |V(G_*^{(k)})| \le kn^2$ and $sp|U| \le s \cdot \frac{2n (\lceil \log k \rceil + 1)^2}{|U|} \cdot |U| = 2sn(\lceil \log k \rceil + 1)^2$. Hence, each element in $C_j$ decomposes into at most $(kn^2)^{2sn(\lceil \log k \rceil + 1)^2}$ elements in the next iteration. Therefore, after $j$ iterations, the number of containers is bounded by $|\mC_j \cup \mD_j| \le (kn^2)^{j \cdot 2sn(\lceil \log k \rceil + 1)^2}$, and each container in $\mC_j$ has size at most $(1-\delta\lambda^{-1})^j kn^2$.

	Choosing $j = O_s\left(\log^2 n \right)$ such that $(1-\delta\lambda^{-1})^j kn^2 \le k (\lceil \log k \rceil + 1) n$, we have $\mC_j = \varnothing$. But each $K_s$-free subset of $G_*^{(k)}$ is contained in some element of $\mC_j \cup \mD_j$. Therefore, the total number of $K_s$-free subsets of size $t$ is at most
	   \[ \sum_{X \in \mD_j} \binom{|X|}{t} \le (kn^2)^{O_s\left( n \log^2 n (\lceil \log k \rceil + 1)^2 \right)} \cdot \binom{k (\lceil \log k \rceil + 1) n}{t} \le (10k)^t \quad \text{for } n \ge C(s), \]
	as claimed.
\end{proof}

For every graph $G$, we say that $G'$ is a \emph{random permutation} of $G$ if $G'$ is obtained by taking a random permutation $\pi$ uniformly from all permutations on $V(G)$ and applying $\pi$ to $V(G)$ such that $\{\pi(i),\pi(j)\} \in E(G')$ if and only if $\{i,j\} \in E(G)$. For graphs $G_1,\cdots, G_j$ on the same vertices set $V$, \emph{overlaying them randomly} means taking random permutations $G'_1,\cdots, G'_j$ of $G_1,\cdots, G_j$ respectively, where the random permutations are independent, and take the union of their edges.

\begin{corollary} \label{cor:overlay}
	Let $G_*$ be an $s$-partite sparsification given by \Cref{lem:good_partition}.
	Let $\beta \ge 0$ be an integer and $k = n^\beta$.
	Consider $\beta+1$ disjoint copies of $G_*^{(k)}$. Overlay them randomly and then retain each vertex independently with probability $p = \frac{1}{(\lceil \log k \rceil + 1)^{2\beta} \cdot (\log n)^{4\beta} \cdot (\log n)^{32\beta \sqrt{\log n}}} = \frac{1}{n^{o(1)}}$. Then, with positive probability, all $K_s$-free subsets of size $t = (\lceil \log k \rceil + 1)^2 \cdot n \log^4 n$ are eliminated\footnote{That is, none are preserved entirely in the resulting subgraph. }.
\end{corollary}

\begin{proof}
	Let $N = |V(G_*^{(k)})| = k |V(G_*)| \ge \frac{kn^2}{(\log n)^{32\sqrt{\log n}}}$. From \Cref{cor:count}, the number of $K_s$-free subsets of size $t$ in each copy of $G_*^{(k)}$ is at most $(10k)^t$.

	Consider $\beta+1$ of random bijections $\phi_1, \cdots, \phi_{\beta+1} \cl [N] \to V(G_*^{(k)})$, and define a graph on the vertex set $[N]$ by connecting $x, y$ if $\phi_i(x)$ and $\phi_i(y)$ are adjacent in $G_*^{(k)}$ for some $i \in [\beta+1]$.

	For a fixed $t$-subset $I \subseteq [N]$, the probability that $I$ is a $K_s$-free set in each of the $\beta+1$ copies is at most $\left( \frac{(10k)^t}{\binom{N}{t}} \right)^{\beta+1}$. Hence, the expected number of $K_s$-free subsets of size $t$ in the resulting union is at most

	\begin{align*}
		  & \binom{N}{t} \cdot \left( \frac{(10k)^t}{\binom{N}{t}} \right)^{\beta+1} = \frac{(10k)^{(\beta+1) t}}{\binom{N}{t}^\beta} \le \frac{(10k)^{(\beta+1) t}}{(N/t)^{\beta t}} \le \frac{(10k)^{(\beta+1) t}}{(kn^2/(\log n)^{32\sqrt{\log n}}t)^{\beta t}} \\
		  & = 10^{(\beta+1)t} \left( \frac{k^{\beta+1}}{(kn)^\beta} \right)^t \cdot \left( (\lceil \log k \rceil + 1)^2 \cdot (\log n)^{32 \sqrt{\log n} + 4} \cdot \right)^{\beta t}                                                                              \\
		  & = 10^{(\beta+1)t} \left( (\lceil \log k \rceil + 1)^2 \cdot (\log n)^{32 \sqrt{\log n} + 4} \right)^{\beta t} \eqcolon \Lambda^t.
	\end{align*}

	Fix a realization of the graph where the total number of $K_s$-free subsets of size $t$ is at most the above quantity. Now, retain each vertex independently with probability $p = \frac{1}{\Lambda}$. Since the probability of a fixed $t$-subset is preserved is $p^t$. By the union bound, with positive probability, all $K_s$-free sets of size $t$ are eliminated.
\end{proof}

Now, we are able to give our proof for \cref{thm:upper_bound}.

\begin{proof}[Proof of \cref{thm:upper_bound}]
	Let $\ell$ be the minimum integer such that $r_\ell(b) > s$. Note that if $\ell > t$, then $\lfloor t/\ell \rfloor = 0$ and \cref{thm:upper_bound} is trivial. Hence, we can assume that $\ell \le t$ and then $\lfloor t/\ell \rfloor \ge 1$.

	Let $\beta = \lfloor t/\ell \rfloor -1$ and $k = n^\beta$. Let $G_*^{(k)}$ be the $k$-blow-up of the graph $G_*$ given by \Cref{lem:good_partition}. Let $H$ be the graph obtained by taking $\beta+1$ copies $G_0,\cdots,G_{\beta}$ of $G_*^{(k)}$, overlaying them randomly and then retaining each vertex independently with probability $p = \frac{1}{(\lceil \log k \rceil + 1)^{2\beta} \cdot (\log n)^{4\beta} \cdot (\log n)^{32\beta \sqrt{\log n}}} = \frac{1}{n^{o(1)}}$. By the standard Chernoff bound, we have
	   \[ \bigl|V(H)\bigr| \ge \Omega\bigl(|V(G_*)| \cdot k \cdot p\bigr) = n^{2-o(1)} \cdot n^{\lfloor \frac{t}{\ell} \rfloor -1} \cdot  \frac{1}{n^{o(1)}} = n^{\lfloor \frac{t}{\ell} \rfloor +1 - o(1)} \]
	with high probability. By \cref{cor:overlay}, all $K_s$-free subsets of size $n^{1+o(1)}$ are eliminated with high probability, so $\alpha_{s}(H) \le n^{1+o(1)}$. Now it suffices to prove that $H$ is \text{($\underbrace{K_b, \cdots, K_b}_t$)}-free.

	We first have that $G_i$ is \text{($\underbrace{K_b, \cdots, K_b}_\ell$)}-free for every $i \in \{0,1,\cdots,\beta\}$. Indeed, we just need to prove the $s$-partite sparsification $G_*$ is \text{($\underbrace{K_b, \cdots, K_b}_\ell$)}-free, since taking the blow-up will not change this property. Note that by our assumption $r_\ell(b) > s$, there is a coloring $\pi$ of $K_s$ with $\ell$ colors such that there is no monochromatic copy of $K_b$. Now, for every set of $s$ parts in $G_*$ belonging to the same clique $K_v$ in the original graph, we color all edges between parts $i$ and $j$ by color $\pi(i,j)$. Then, there is no monochromatic copy of $K_b$ within $G_*[K_v]$. Also, recall that by \cref{prop:EFR_constr} (iii), every copy of $K_b$, $b\ge 3$, is contained entirely in some $K_v$. Therefore, there is no monochromatic copy of $K_b$ in $G_*$.

	Now, we can use at most $(\beta+1)\cdot \ell \le t$ colors to color $G_0,\cdots, G_\beta$ such that they use disjoint sets of colors and there is no monochromatic copy of $K_b$ in any $G_i$. When overlaying these graphs, we can specify an arbitrary color for the joint edges and note that the resulting graph still does not contain any monochromatic copy of $K_b$, so $H$, an induced subgraph of the resulting graph, is \text{($\underbrace{K_b, \cdots, K_b}_t$)}-free. \qedhere
\end{proof}

\section{Recursive Lower Bounds} \label{sec:lower_bound}
In this section, we prove \cref{thm:lower_bound}, our recursive lower bound for the multicolor Erd\H{o}s--Rogers functions in the setting of triangle-free colorings.

Sudakov~\cite{Sud05b} established the recursive lower bound $f_{s,t}(n) \ge \Omega(n^{a'_t})$ for the classical Erd\H{o}s--Rogers function, where the exponent $a'_t$ satisfies the recurrence
\begin{equation} \label{eq:rec_formula_Sud}
	a'_t = 1 \quad \text{for } 1 \le t \le s, \quad \text{and} \quad \frac{1}{a'_t} = 1 + \frac{1}{s-1} \sum_{i = 1}^{s-1} \frac{1}{a'_{t-i}} \quad \text{for } t \ge s+1.
\end{equation}
In a subsequent work~\cite{Sud05a}, Sudakov refined the initial value $a'_{s+2}$, obtaining the lower bound $f_{s,s+2}(n) \ge \Omega(n^{\frac{3s-4}{6s-6}})$, and thereby improving $f_{s,t}(n)$ for all $t \ge s+2$. In a similar spirit, we establish a recursive lower bound, \cref{thm:lower_bound}, for the multicolor Erd\H{o}s--Rogers function $f_{K_s; K_3, \cdots, K_3}(n)$. We will first prove our recursive formula in \cref{pro:rec_formula} and then give the initial values of the exponent in Lemmas~\ref{lem:lay1},~\ref{lem:lay2}, and~\ref{lem:lay3}. \cref{thm:lower_bound} is a direct corollary of these results.

Recall that we defined the multicolor Ramsey number $r_{t} (b)$ and the local Ramsey number $r^{\text{loc}}_k(H)$ in \cref{sec:int}, and we let $g \cl \N \to \N$ be the inverse function of the local Ramsey number, that is, $g(i)$ is the smallest integer $k$ such that $r^{\text{loc}}_k(3) > i$.

\begin{proposition} \label{pro:rec_formula}
	Let $s,t$ be positive integers and $a_0,\cdots, a_{t-1}$ be a sequence of real numbers\footnote{In the recursive formula, the indices of the sequence $\{a_i\}$ are always non-negative. This is because $r^{\text{loc}}_t(3) \ge r_t(3) > s$, and hence the smallest $k$ such that $r^{\text{loc}}_k(3) > s$ is at most $t$. }. 
    If $r_t(3) > s$ and
	   \[ \ferj (n) \ge \Omega(n^{a_j}) \quad \textrm{for every $j < t$}, \]
	then
	   \[ \fert (n) \ge \Omega(n^{a_t}) \quad \textrm{where} \quad \frac{1}{a_t} = 1 + \frac{1}{s-1} \sum_{i = 2}^{s} \frac{1}{a_{t-g(i)}}. \]
\end{proposition}

Fix a graph $G$, and consider the associated $s$-uniform hypergraph $\mG$, whose hyperedges correspond to the copies of $K_s$ in $G$. Our goal thus reduces to finding a large independent set in $\mG$. To this end, we need the following well-known Tur\'an-type bound on the size of the largest independent set in a $k$-uniform hypergraph. This estimate was also employed in~\cite{Sud05a, Sud05b}. It was established via the alteration method: we select each vertex independently with probability $\Theta((n/m)^{k-1})$, and then remove a vertex from each copy of $K_s$ in the resulting graph.

\begin{lemma} \label{lem:ind_set}
	Let $\mG$ be a $k$-uniform graph with $n$ vertices and $m$ edges. Then $\mG$ contains an independent set of size $\Omega \big(n^{1+\frac{1}{k-1}} \big/ m^{\frac{1}{k-1}}\big)$.
\end{lemma}

Trivially, the number of copies of $K_s$ in $G$ is bounded by the product of codegrees:
\begin{equation} \label{eq:triv_prod_bound}
	e(\mG) \le n \cdot \Delta_1(\mG) \cdot \Delta_2(\mG) \cdots \Delta_{s-1}(\mG).
\end{equation}
If $G$ does not contain a large $K_s$-free subset, we can bound each term on the right-hand side of~\eqref{eq:triv_prod_bound} by induction on the number of colors. In combination with \Cref{lem:ind_set}, this yields a dichotomy for the size of the largest $K_s$-free subset in $G$.
This approach was adopted in~\cite{Sud05b} in the study of uncolored Erd\H{o}s--Rogers function, leading to the recursive formula~\eqref{eq:rec_formula_Sud}. In the multicolor setting, the following lemma provides a way to label the vertices in each copy of $K_s$ in $G$ such that, when constructing the clique according to this ordering, the number of choices at each step is bounded. This leads to a sharper estimate than~\eqref{eq:triv_prod_bound}.

\begin{lemma} \label{lem:order}
	Let $s$ be an integer, and suppose the edges of $K_s$ are colored such that there is no monochromatic triangle. Then there exists an ordering of the vertices $v_1, v_2, \cdots, v_s$ of $K_s$ such that for each $i = 2,3, \cdots, s$, the set of edges $\{v_1v_i, v_2v_i, \cdots, v_{i-1}v_i\}$ uses at least $g(i)$ distinct colors, where $g$ is the inverse function of the local Ramsey number.
\end{lemma}

\begin{proof}
	For every ordering $\pi$ of the vertices $v_1,v_2,\cdots, v_s$ of $K_s$ and every integer $i \ge 2$, let $\ell_\pi(i)$ be the number of distinct colors among the edges $\{v_jv_i \cl j < i\}$ and $n_\pi$ be the number of the pairs of indices $(i,j)$ such that $i<j$ and $\ell_\pi(i) > \ell_\pi(j)$.

	Now fix $\pi$ to be the ordering which minimizes $n_\pi$. We first claim that $n_\pi = 0$, which implies that $\ell_\pi$ is non-decreasing. If not, then there is an index $2\le i \le s-1$ such that $\ell_\pi(i) > \ell_\pi(i+1)$. We swap the positions of $v_i$ and $v_{i+1}$ and get a new ordering $\pi'$. Then by definition, we have $n_{\pi'} < n_{\pi}$, a contradiction.

	We now claim that $\pi$ satisfies the desired condition in the lemma. For each index $i$, consider the subgraph induced by $\{v_1, v_2, \cdots, v_i\}$. Since $\ell_\pi$ is non-decreasing, the coloring restricted to this subgraph is an $\ell_\pi(i)$-local coloring, that is, each vertex is incident to edges of at most $\ell_\pi(i)$ distinct colors. As the original coloring contains no monochromatic triangle, it follows from the definition of $g(i)$ that $\ell_\pi(i) \ge g(i)$, completing the proof.
\end{proof}

\begin{proof}[Proof of \Cref{pro:rec_formula}]
	Suppose $r_t(3) > s$, and define $\alpha$ by
	   \[ \frac{1}{\alpha} = 1 + \frac{1}{s-1} \sum_{i = 2}^{s} \frac{1}{a_{t-g(i)}}. \]

	If for some $i \in \{2, \cdots, s\}$, there exists $i-1$ distinct vertices $v_1, \cdots, v_{i-1}$ such that
	\begin{equation} \label{eq:rec_fml_eq1}
		\Big| \big\{ u \in V(G) \cl uv_1, \cdots, uv_{i-1} \in E(G) \text{  and have at least $g(i)$ distinct color} \big\} \Big| \ge n^{\alpha / a_{t-g(i)}},
	\end{equation}
	then by the pigeonhole principle, there exists a color pattern $(c_1, \cdots, c_{i-1}) \in [t]^{i-1}$ involving at least $g(i)$ distinct colors such that
	   \[ \Big| \big\{ u \in V(G) \cl uv_j \in E(G) \text{ has color } c_j \text{ for every } j \in [i-1] \big\} \Big| \ge \frac{1}{t^{i-1}} \cdot n^{\alpha / a_{t - g(i)}}. \]
	Within these vertices, none of the edges can use any of the colors $c_j$, as that would form a monochromatic triangle. Hence, this subset spans edges using at most $t-g(i)$ colors. By the induction hypothesis, it contains a $K_s$-free subset of size $\Omega \left( \left( \frac{1}{t^{i-1}} \cdot n^{\alpha / a_{t-g(i)}} \right)^{a_{t-g(i)}} \right) = \Omega(n^\alpha)$.

	Suppose \eqref{eq:rec_fml_eq1} does not hold for any $i$ and any set of $i-1$ vertices $v_1, \cdots, v_{i-1}$.
	By \Cref{lem:order}, each copy of $K_s$ in $G$ can be constructed via an ordering $v_1, \cdots, v_s$ such that each $v_i$ belongs to the set described in~\eqref{eq:rec_fml_eq1}. Let $\mG$ be the $s$-uniform hypergraph whose hyperedges are the copies of $K_s$ in $G$. We have that the number of hyperedges in $\mG$ is at most
	   \[ m \le n \cdot \prod_{i = 2}^s n^{\alpha / a_{t-g(i)}}. \]
	By \Cref{lem:ind_set}, the hypergraph $\mG$ contains an independent set of size
	   \[ \Omega \left(n^{1+\frac{1}{s-1}} \big/ m^{\frac{1}{s-1}}\right) = \Omega \left( n / n^{\frac{1}{s-1} \sum_{i = 2}^{s} \alpha / a_{t-g(i)}} \right) = \Omega(n^\alpha), \]
	so $\alpha_s(G) = \Omega(n^{\alpha})$.
\end{proof}

In the following lemmas, we establish the initial values of the exponent in \cref{thm:lower_bound}.
\begin{lemma} \label{lem:lay1}
	For integers $s \ge 2$ and $t \ge 1$ such that $r_t(3) \le s$, if a graph $G$ is $(\underbrace{K_3,\cdots, K_3}_t)$-free, then $G$ is $K_s$-free. Hence, we have
	   \[ \fert (n) = n. \]
\end{lemma}
\begin{proof}
	Suppose for contradiction that $G$ contains a copy of $K_s$. By the definition of $r_t(3)$, we have that in every $t$-edge coloring of $G$, this copy of $K_s$ contains a monochromatic triangle, a contradiction.
\end{proof}

\begin{lemma} \label{lem:lay2}
	For integers $s \ge 2$ and $t \ge 1$ such that $r_{t-1}(3) \le s < r_t(3)$, we have
	   \[ \fert (n) \ge \Omega(n^{1/2}). \]
\end{lemma}
\begin{proof}
	By \cref{pro:rec_formula} and \cref{lem:lay1}, we have
	   \[ 1 + \frac{1}{s-1} \sum_{i = 2}^{s} \frac{1}{a_{t-g(i)}} = 1 + \frac{1}{s-1} \sum_{i = 2}^{s} 1 = 2, \]
	which verifies our claim.
\end{proof}

For the next level in the Ramsey hierarchy, namely when $r_{t-2}(3) \le s < r_{t-1}(3)$, we have the following lower bound, which is sharper than trivially using \cref{pro:rec_formula} and Lemmas~\ref{lem:lay1} and~\ref{lem:lay2}.
\begin{lemma} \label{lem:lay3}
	For integers $s \ge 3$ and $t \ge 1$ such that $r_{t-2}(3) \le s < r_{t-1}(3)$, we have
	   \[ \fert \ge \Omega\Big(n^{\frac{s + \lceil \frac{s}{2} \rceil - 3}{2s + 2 \lceil \frac{s}{2} \rceil - 5}} \Big). \]
\end{lemma}

We first need the following lemma, which extends \Cref{lem:ind_set} to the setting of hypergraphs with different uniformity.
\begin{lemma} \label{lem:new_ind_set}
	Let $s, s' \ge 2$ be integers, and let $\mF$ and $\mG$ be $s$-uniform and $s'$-uniform hypergraphs on $[n]$, respectively. Then there exists a subset $I \subseteq [n]$ of size $\Omega \left( \min \left\{ \frac{n^{s/(s-1)}}{|\mF|^{1/(s-1)}}, \frac{n^{s'/(s'-1)}}{|\mG|^{1/(s'-1)}} \right\} \right)$ such that $I$ is an independent set in both $\mF$ and $\mG$.
\end{lemma}

\begin{proof}
	Let $p = \frac{1}{3} \cdot \min \left\{ \frac{n^{1/(s-1)}}{|\mF|^{1/(s-1)}}, \frac{n^{1/(s'-1)}}{|\mG|^{1/(s'-1)}} \right\}$. Select each vertex independently with probability $p$, and then remove one vertex from each edge of $\mF$ or $\mG$ that is entirely contained in the selected set. The expected number of selected vertices is $pn$, and the expected number of edges of $\mF$ and $\mG$ that are fully contained in the selected set is $e(\mF) \cdot p^s$ and $e(\mG) \cdot p^{s'}$, respectively.

	Therefore, the expected number of remaining vertices, which form an independent set in both $\mF$ and $\mG$, is at least:
	\begin{align*}
		pn - e(\mF) \cdot p^s - e(\mG) \cdot p^{s'} & \ge pn - e(\mF) \cdot \frac{n^{s/(s-1)}}{3^s \cdot e(\mF)^{s/(s-1)}}  - e(\mG) \cdot \frac{n^{s'/(s'-1)}}{3^{s'} \cdot e(\mG)^{s'/(s'-1)}} \\
		                                            & \ge pn - \frac{pn}{3} - \frac{pn}{3}
		\ge \frac{pn}{3}. \qedhere
	\end{align*}
\end{proof}

\begin{proof}[Proof of \cref{lem:lay3}]
	For each edge $xy$,
	we say that $y$ is a neighbor of $x$ of the color $c$ if the edge $xy$ has color $c$. For each vertex $x$, let $\dst(x)$ and $\dnd(x)$ denote the sizes of the largest and the second largest color classes in the neighborhood of $x$, respectively.
	Define $\alpha \ce \frac{s+\lceil \frac{s}{2} \rceil - 3}{2s + 2\lceil \frac{s}{2} \rceil - 5}$ and $\beta \ce \alpha + 1$.

	If $\mathbb{E}_{x \in V(G)} (\dst(x) \cdot \dnd(x)) \ge n^\beta$, there are at least $n \cdot n^\beta$ triples $(x,y,z)$ such that $xy$ and $xz$ are connected and of different colors. Since there are at most $t(t-1)$ possible color pairs for $(xy, xz)$, one such pair must occur in at least $\frac{n^{1+\beta}}{t(t-1)}$ times. Without loss of generality, assume that the color pair $(\text{red}, \text{blue})$ appears most frequently.

	By averaging over all $(y,z)$ pairs, there exist distinct vertices $y \neq z$ such that
	   \[ \Big| \big\{ x \in V(G) \cl xy \text{ is red and } xz \text{ is blue} \big\} \Big| \ge \frac{n^{1+\beta}}{t(t-1)n^2} = \frac{n^{\alpha}}{t(t-1)}. \]
	Since each color is triangle-free, the induced subgraph on this set contains no edges of color red and blue. Hence, this set only spans edges of the remaining $t-2$ colors. Given that $r_{t-2}(3) \le s$, this implies the above set is $K_s$-free.

	Thus, we may assume $\mathbb{E}_{x \in V(G)} (\dst(x) \cdot \dnd(x)) < n^\beta$. We now construct an $s$-uniform set family $\mF$ and an $\lceil \frac{s}{2} \rceil$-uniform set family $\mG$ such that every copy of $K_s$ in $G$ either lies in $\mF$, or contains some member of $\mG$ as a subset. We need the following claim.

	\begin{claim}
		Let $k \ge 2$ be an integer. In any copy of $K_k$ in $G$, there exists $\lceil \frac{k}{2} \rceil$ vertices $v_1, v_2, \cdots, v_{\lceil \frac{k}{2} \rceil}$ such that:
		\vspace{-1ex}\begin{enumerate}\setlength\itemsep{-.8ex}
			\item The edge $v_1v_2$ does not have the most frequent color among the edges incident to $v_1$.
			\item For each $i = 3, \cdots, \lceil \frac{k}{2} \rceil$, the set of edges $\{v_1v_i, \cdots, v_{i-1}v_i\}$ contains at least two distinct colors.
		\end{enumerate}
	\end{claim}
	\begin{poc}
		The case $k = 2$ is trivial. We proceed by induction on $k$ for $k \ge 3$. Let $c(xy)$ denote the color of the edge $xy$. For each vertex $x$, define the \emph{most frequent color at $x$} to be the color that appears most frequently among the edges incident to $x$.

		For $k = 3$, consider an arbitrary triangle $xyz$ in $G$. Since each color is triangle-free, the colors $xy, xz$, and $yz$ cannot be the same. We may assume that $xy$ and $xz$ have different colors. Then at least one of them is not the most frequent color among the edges incident to $x$, say $xy$. We then take $v_1 = x$ and $v_2 = y$.

		Now assume the claim holds for $k-1$. If $\lceil \frac{k-1}{2} \rceil = \lceil \frac{k}{2} \rceil$, i.e., $k$ is even, the induction hypothesis gives us the desired result directly. So we may assume $k$ is odd.

		Let $U \subseteq V(G)$ be a clique of size $k$. Remove an arbitrary vertex from $U$ and apply the induction hypothesis to the remaining $K_{k-1}$. We obtain a sequence of $\lceil \frac{k}{2} \rceil - 1$ vertices $v_1, \cdots, v_{ \lceil \frac{k}{2} \rceil - 1}$ satisfying the claim. Let $V = \{ v_1, \cdots, v_{ \lceil \frac{k}{2} \rceil - 1} \}$ and set $W = U \setminus V$. Then $|W| = k - |V| = |V| + 1$.

		If there exists a vertex $w \in W$ such that the set of edges $\{wv_1, \cdots, wv_{|V|}\}$ uses at least two distinct colors, then we may take $v_{|V| + 1} = w$ to complete the induction. Otherwise, suppose that for every $w \in W$, all edges between $w$ and $V$ are of the same color. Denote this common color by $c(w)$.

		If for some $w \in W$, the color $c(w)$ is not the most frequent color at $w$, then the sequence $w, v_1, \cdots, v_{|V|}$ satisfies the claim. Indeed, for each $i \ge 2$, both $wv_1$ and $wv_i$ have color $c(w)$, and since $G$ contains no monochromatic triangle, $v_1v_i$ must be a different color. Thus, the edges incident to $v_i$ among earlier vertices use at least two colors.

		So we may assume that $c(w)$ is the most frequent color at $w$ for all $w \in W$.
		\vspace{-1ex}\begin{itemize}\setlength\itemsep{-.8ex}
			\item If there exist $w_1, w_2 \in W$ such that $c(w_1) \neq c(w_2)$, consider the color of the edge $w_1w_2$. Since $c(w_1) \neq c(w_2)$, at least one of them differs from $c(w_1w_2)$, say, $c(w_1w_2) \neq c(w_1)$. Then $c(w_1w_2)$ is not the most frequent color at $w_1$. Therefore, the sequence $w_1, w_2, v_1, \cdots, v_{|V|-1}$ satisfies the claim.
			\item If all $c(w)$'s are equal, say to some color $\gamma$. Since $\gamma$ is triangle-free, no edge within $W$ has color $\gamma$. Therefore, for any edge in $G[W]$, its color differs from the most frequent color at both endpoints. In this case, the desired sequence can be obtained by applying \Cref{lem:order} to $W$. \qedhere
		\end{itemize}
	\end{poc}

	Take $k = s$ in the claim and write $s' = \lceil \frac{s}{2} \rceil$. For each vertex $x$, define $\xi \ce \dst(x) \cdot \dnd(x)$ and set $\ell \ce \frac{\xi^{(s-1)/(s+s'-2)}}{n^{\alpha(s-s')/(s+s'-2)}}$. We construct families $\mF_x$ and $\mG_x$ locally at each vertex $x$ as follows:

	\vspace{-.5ex}\begin{itemize}\setlength\itemsep{-.5ex}
		\item If $\dst(x) \le \ell$, let $\mF_x$ consist of all $s$-tuples $\{x, v_2, \cdots, v_s\}$ such that:
		      \vspace{-1ex}\begin{itemize}\setlength\itemsep{-.4ex}
			      \item $\{x, v_2, \cdots, v_s\}$ induces a clique in $G$;
			      \item For each $i = 3, \cdots, s$, the edges $\{xv_i, v_2v_i, \cdots, v_{i-1}v_i\}$ have at least two distinct colors.
		      \end{itemize}\vspace{-2ex}
		      and set $\mG_x = \varnothing$.
		\item If $\dst(x) > \ell$, then $\dnd(x) \le \xi / \ell$. Let $\mG_x$ consist of all $s'$-tuples $\{x, v_2, \cdots, v_{s'}\}$ such that:
		      \vspace{-1ex}\begin{itemize}\setlength\itemsep{-.4ex}
			      \item $\{x, v_2, \cdots, v_{s'}\}$ induces a clique in $G$;
			      \item $xv_2$ is not of the most frequent color among the edges incident to $x$;
			      \item For each $i = 3, \cdots, s'$, the edges $\{xv_i, v_2v_i, \cdots, v_{i-1}v_i\}$ have at least two distinct colors.
		      \end{itemize}\vspace{-2ex}
		      and set $\mF_x = \varnothing$.
	\end{itemize}\vspace{-1ex}

	In the first case, the color of $xv_2$ has $t$ choices, and for each color, there are at most $\dst(x) \le \ell$ neighbors of $X$. Therefore, there are at most $t \cdot \ell$ choices of $v_2$. For each color pattern $(c_1, \cdots, c_{i-1}) \in [t]^{i-1}$ and fixed $v_2, \cdots, v_{i-1}$, the number of $v_i$ satisfying the conditions is at most $n^\alpha$, for otherwise we obtain a $K_s$-free subset of size at least $n^\alpha$. Thus, the size of $\mF_x$ is at most
	   \[ \bigl|\mF_x\bigr| \le t \cdot \ell \cdot \prod_{i = 3}^s (t^{i-1} \cdot n^\alpha) \le t \cdot \ell \cdot (t^s \cdot n^\alpha)^{s-2} = O \left( \xi^{(s-1)/(s+s'-2)} \cdot n^{\frac{(s-1)(s+s'-4)}{s+s'-2} \cdot \alpha} \right). \]

	In the second case, by a similar argument, the size of $\mG_x$ is at most
	   \[ \bigl|\mG_x\bigr| \le (t-1) \cdot \frac{\xi}{\ell} \cdot (t^{s'} \cdot n^\alpha)^{s'-2} = O \left( \xi^{(s'-1)/(s+s'-2)} \cdot n^{\frac{(s'-1)(s+s'-4)}{s+s'-2} \cdot \alpha} \right). \]

	Since $\E (\xi) \le n^\beta$, we have
	\begin{gather*}
		\E\bigl(\bigl|\mF_x\bigr|\bigr) \le \E \left( O \left( \xi^{(s-1)/(s+s'-2)} \cdot n^{\frac{(s-1)(s+s'-4)}{s+s'-2} \cdot \alpha} \right) \right) \le O\left( n^{\frac{s-1}{s+s'-2} \cdot \beta} \cdot n^{\frac{(s-1)(s+s'-4)}{s+s'-2} \cdot \alpha} \right), \\
		\E\bigl(\bigl|\mG_x\bigr|\bigr) \le \E\left( O \left( \xi^{(s'-1)/(s+s'-2)} \cdot n^{\frac{(s'-1)(s+s'-4)}{s+s'-2} \cdot \alpha} \right) \right) \le O\left( n^{\frac{s'-1}{s+s'-2} \cdot \beta} \cdot n^{\frac{(s'-1)(s+s'-4)}{s+s'-2} \cdot \alpha} \right).
	\end{gather*}

	Let $\mF = \bigcup \mF_x$ and $\mG = \bigcup \mG_x$. By \Cref{lem:new_ind_set}, there exists an independent set of size
	   \[ \Omega \left(n \big/ (n^{\frac{1}{s+s'-2} \cdot \beta} \cdot n^{\frac{s+s'-4}{s+s'-2} \cdot \alpha}) \right) = \Omega \left( n^{\frac{s+s'-3}{s+s'-2}} / n^{\frac{s+s'-3}{s+s'-2} \cdot \alpha} \right) = \Omega( n^\alpha ). \qedhere \]
\end{proof}

Finally, we verify the lower bounds in~\cref{cor:examples}.
\begin{proof}[Proof of lower bounds in~\cref{cor:examples}]
	Let $s = 5$. Then $r_1(3) = 3 \le s=5 < r_2(3)$.

	When $t=2$, by~\cref{lem:lay2}, $a_2=\frac{1}{2}$ and $f_{\alpha_5; K_3,  K_3}(n)\ge \Omega(n^{1/2})$.

	When $t=3$, \Cref{lem:lay3} yields
	   \[ f_{\alpha_5; K_3, K_3, K_3}(n) \,\ge\, \Omega \left( n^{(s+\lceil \frac{s}{2} \rceil - 3)/(2s + 2\lceil \frac{s}{2} \rceil - 5)} \right) \,=\, \Omega(n^{5/11}) \,\eqcolon\, \Omega(n^{a_3}). \]

	When $t=4$, by \Cref{pro:rec_formula}, we have
	   \[ f_{\alpha_5; K_3, K_3, K_3, K_3}(n) \ge \Omega(n^{a_4}), \]
	where
	   \[ \frac{1}{a_4} \,=\, 1 + \frac{1}{4} \left( \frac{1}{a_3} + \frac{1}{a_2} + \frac{1}{a_2} + \frac{1}{a_2} \right) \,=\, 1 + \frac{1}{4} \left( \frac{11}{5} + 2 + 2 + 2 \right) \,=\, \frac{61}{20}. \qedhere\]
\end{proof}

\section{Concluding Remarks} \label{sec:con}
In this paper, we investigated the maximum size of a $K_s$-free subset that can be found in a $(K_b,\cdots,K_b)$-free graph.

For the case of $f_{\alpha_5; K_3, K_3, K_3}(n)$, our bounds suggest that its growth is fairly close to $\widetilde{\Theta}(n^{1/2}) = f_{\alpha_5; K_3, K_3}(n)$. It would be interesting to determine its exact growth. For $f_{\alpha_5; K_3, K_3, K_3}(n)$, the exponent is in $[5/11,1/2]$, we think the main task is to construct examples in which the third color is utilized more effectively, thereby separating it from the two-color case.

In the analysis of $f_{\alpha_5; K_3, K_3, K_3, K_3}(n)$ in \Cref{pro:rec_formula}, the dominant contribution to the number of $K_5$'s in the graph $G$ arises from the set of two-colored $K_5$'s formed as the union of two monochromatic $C_5$'s, as other colorings of $K_5$ can be shown to contribute at most $o(n \cdot \alpha_5(G)^8)$. A more refined estimate on the number of such $K_5$'s would lead to an improved bound on $f_{\alpha_5; K_3, K_3, K_3, K_3}(n)$. It might even be possible to determine the asymptotic value without first resolving the previous case $f_{\alpha_5; K_3, K_3, K_3}(n)$.

\section*{Acknowledgment}
This work was initiated during the $2^{\text{nd}}$ ECOPRO Student Research Program at the Institute for Basic Science (IBS) in the summer of 2024. Luo and Ouyang are very grateful for the kind hospitality of IBS.

\end{document}